\documentclass[a4paper]{amsart}
\usepackage{tikz-cd}
\usepackage{amsmath,amsthm,amssymb,amscd}
\usepackage{mathtools}

\usepackage{enumitem}
\setlist[enumerate,1]{label=(\roman*), font=\normalfont}

\newcommand{\Z}{\mathbb{Z}}

\newcommand{\R}{\mathbb{R}}

\newcommand{\N}{\mathbb{N}}

\renewcommand{\Z}{\mathbb{Z}}

\newcommand{\Top}{\mathrm{T\ddot{o}p}}

\newtheorem{theorem}{Theorem}[section]
\newtheorem{lemma}[theorem]{Lemma}

\newtheorem{proposition}[theorem]{Proposition}
\newtheorem{claim}{Claim}

\theoremstyle{definition}
\newtheorem{definition}[theorem]{Definition}

\newtheorem{question}[theorem]{Question}

\newtheorem{remark}[theorem]{Remark}

\newcommand{\actson}{\curvearrowright}
\newcommand{\sub}{\subseteq}
\DeclareMathOperator{\Per}{Per}
\DeclareMathOperator{\Sym}{Sym}
\DeclareMathOperator{\diam}{diam}
\DeclarePairedDelimiter\set{\{}{\}}
\DeclarePairedDelimiter\nm{\lVert}{\rVert}
\newcommand{\eps}{\epsilon}
\newcommand{\eqrel}[1]{\,#1\,}
\newcommand{\cl}[2][]{\overline{#2}^{#1}}
\newcommand{\cS}{\mathcal{S}}
\renewcommand{\And}{\text{ and }}

\author{Marcin Sabok}
\address{Marcin Sabok, Department of Mathematics and Statistics, McGill
  University, 805, Sherbrooke Street West Montreal, Quebec,
  Canada H3A 2K6 and Institute of Mathematics, Polish
  Academy of Sciences, \'Sniadeckich 8, 00-655 Warszawa,
  Poland} 
\email{marcin.sabok@mcgill.ca}

\author{Todor Tsankov}
\address{Todor Tsankov, Institut de Math\'ematiques de Jussieu--PRG,
  Universit\'e Paris Diderot, 75205 Paris \textsc{cedex} 13}
\email{todor@math.univ-paris-diderot.fr}

\thanks{The first author was partially supported by NSERC and the
  Polish Ministry of Science and Higher Education (MNiSW) through the
  grant Mobilno\'s\'c Plus. The second author was partially supported
  by the ANR contracts GrupoLoco (ANR-11-JS01-008) and GAMME
  (ANR-14-CE25-0004).}

\title[Topological conjugacy of Toeplitz subshifts]{On the complexity of topological conjugacy of Toeplitz subshifts}

\begin{document}

\begin{abstract}
  In this paper, we study the descriptive set theoretic complexity of
  the equivalence relation of conjugacy of Toeplitz subshifts of a
  residually finite group $G$. On the one hand, we show that if
  $G = \Z$, then topological conjugacy on Toeplitz subshifts with
  separated holes is amenable. In contrast, if $G$ is non-amenable,
  then conjugacy of Toeplitz $G$-subshifts is a non-amenable
  equivalence relation. The results were motivated by a general
  question, asked by Gao, Jackson and Seward, about the complexity of
  conjugacy for minimal, free subshifts of countable groups.
\end{abstract}

\maketitle

\section{Introduction}

The theory of definable equivalence relations offers tools for
classifying the complexity of equivalence relations arising from
various isomorphism problems. An important class of Borel equivalence
relations is given by the \emph{countable} ones, i.e., those with
countable classes. It is a classical result of Feldman and Moore that
every such equivalence relation arises as the orbit equivalence
relation of a Borel action of a countable group; and thus, from the
beginning, the theory is intimately connected to that of dynamical
systems from where it has borrowed most of its tools and techniques.
The descriptive set theoretic approach to those equivalence relations
has been developed over the last twenty years by Dougherty, Jackson,
Kechris, Louveau, Hjorth, Thomas and others (see, e.g., \cite{jkl,djk,hk}).

The natural comparison of equivalence relation is given by \emph{Borel
  reducibility} and the simplest ones are those which are
\textit{smooth}, i.e., admit real numbers as complete invariants (or,
equivalently, admit a Borel transversal). The next level of the
complexity hierarchy is formed by the \textit{hyperfinite} ones,
equivalently, those induced by Borel actions of the group of integers
$\Z$. There is also a \emph{universal} countable Borel equivalence
relation, which is maximal in the quasi-order of Borel reducibility,
an example is the orbit equivalence relation $F_2 \actson 2^{F_2}$.

In topological dynamics, one of the most commonly considered types of
dynamical systems are the \emph{subshifts} (also known as
\textit{Bernoulli subflows} or \emph{symbolic dynamical systems}). For
a countable group $G$, the \emph{Bernoulli shift} is the action of $G$
on $2^G$ defined by the formula: $(g \cdot x)(h) = x(g^{-1}h)$ for all
$g, h \in G$, $x \in 2^G$. A \textit{$G$-subshift} is a closed
nonempty subset $S \subseteq 2^G$ which is invariant under the action
of $G$. The natural isomorphism relation of subshifts is
\emph{topological conjugacy}: two subshifts $S$ and $T$ are
\textit{conjugate} if there is a homeomorphism $f \colon S \to T$
which commutes with the action of $G$. This is the equivalence
relation that we study in this paper. Often, special classes of
subshifts are of interest in dynamics. A subshift $S$ is
\textit{minimal} if it does not contain any proper subshift, or
equivalently, if every orbit is dense. It is \textit{free} if for any
$x\in S$ and any $g\in G$ different from $1_G$, we have
$g \cdot x \neq x$.

It is worth noting that while many countable equivalence relations
arise naturally from group actions, some do not, and their study is
usually more difficult because most of the available tools are
dynamical in nature and require the presence of a (natural) group
action. Some notable examples where there is no natural group action
that gives the equivalence relation are Turing equivalence,
isomorphism of (various classes of) finitely generated groups, and
topological conjugacy of subshifts. (However, for subshifts, there is
still a group action present that can be exploited: see
Section~\ref{sec:non-amenable-case}.)

The recent monograph of Gao, Jackson, and Seward \cite{gjs} studies
the complexity of topological conjugacy of free, minimal subshifts.
(In fact, before their construction, it was an open problem whether
such subshifts necessarily exist for every countable $G$.) It follows
essentially from a classical result of Curtis, Hedlund and Lyndon (see
\cite{lind.marcus} or \cite[Lemma 9.2.1]{gjs}) that for any countable
group $G$, topological conjugacy of $G$-subshifts is a countable Borel
equivalence relation. Gao, Jackson and Seward showed \cite[Corollary
1.5.4]{gjs} that if $G$ is infinite, this equivalence relation is not
smooth, and that if $G$ is locally finite, then it is hyperfinite
\cite[Theorem 1.5.6]{gjs}. They also pose the general question
\cite[Problem 9.4.11]{gjs} to determine the complexity of this
equivalence relation for an arbitrary countable group $G$. This was an
important motivating question for our work and in
Theorem~\ref{th:non-amenable}, we provide a partial answer.

Clemens~\cite{clemens} proved that the topological conjugacy of
$\Z$-subshifts is a universal countable Borel equivalence relation.
However, his construction produces subshifts that are far from minimal
and it remains an open question whether isomorphism of minimal
$\Z$-subshifts is universal. This question is connected with the
conjecture of Thomas \cite[Conjecture 1.2]{thomas} that isomorphism of
finitely generated, amenable, simple groups is universal. It follows
from the results of Matui~\cite{matui}, Giordano, Putnam,
Skau~\cite{gps} and a recent result of Juschenko and Monod~\cite{jm}
that the computation of the topological full group of a minimal
$\Z$-subshift provides a reduction from the (flip-) conjugacy of
minimal $\Z$-subshifts to isomorphism of finitely generated, simple,
amenable groups. Another related result was recently proved by
Williams~\cite{jay}, who showed that isomorphism of finitely generated,
solvable groups is weakly universal.

The focus of this paper is studying the complexity of the conjugacy
equivalence relation for the class of \emph{Toeplitz subshifts} of
residually finite groups. This class of subshifts is well-known and
appears in many contexts. For example, Downarowicz~\cite{downarowicz}
showed that that any Choquet simplex can be realized as the simplex of
invariant measures of a Toeplitz subshift. (This result was recently
generalized to arbitrary amenable, residually finite groups by Cortez
and Petite~\cite{Cortez2014}.) An important feature of Toeplitz words
is that they can be constructed in stages, which allows a fair amount
of control. We briefly recall the classical definition for $G = \Z$ here and
postpone the general one for residually finite groups to
Section~\ref{sec:Toeplitz-subshifts}.

A word $x \in 2^\Z$ is \emph{Toeplitz} if every symbol occurs periodically,
i.e., for every $n \in \N$, there exists $k$ such that $x(n + ki) =
x(n)$ for all $i \in \Z$. A subshift $S \sub 2^\Z$ is \emph{Toeplitz}
if it is equal to the closure of the orbit of some Toeplitz word. It
is easy to check that every Toeplitz subshift is minimal.

The topological conjugacy relation for Toeplitz $\Z$-subshifts has
been studied by Downarowicz, Kwiatkowski and Lacroix \cite{dkl} and it
essentially follows from their results that topological conjugacy of
\emph{pointed} Toeplitz flows (i.e., the relation $E$ on Toeplitz
words, such that $x\mathrel{E}y$ if there is a an $\Z$-equivariant
homeomorphism from $\overline{\Z \cdot x}$ to $\overline{\Z \cdot y}$
that maps $x$ to $y$) is a hyperfinite equivalence relation. The
latter seems to indicate that the topological conjugacy relation for
Toeplitz $\Z$-subshifts should not be too complicated. However, we
have only been able to treat the case of Toeplitz subshifts \emph{with
  separated holes}, a special but important class. The first of our
main results is the following.
\begin{theorem}
  \label{th:z-subshifts}
  For $G = \Z$, the equivalence relation of conjugacy on Toeplitz
  subshifts \emph{with separated holes} is amenable and therefore
  hyperfinite $\mu$-a.e. for every Borel probability measure $\mu$ on
  the set of subshifts.
\end{theorem}
We postpone the definitions to Section~\ref{sec:z-subshifts} and just
recall that an amenable equivalence relation is hyperfinite a.e. with
respect to any probability measure but it is an open problem whether
it must be hyperfinite everywhere. We do not know whether the above
equivalence relation is hyperfinite.

The second part of the paper deals with Toeplitz subshifts of
residually finite, non-amenable groups. For those, we prove that
topological conjugacy is somewhat complicated.
\begin{theorem}
  \label{th:non-amenable}
  If $G$ is a non-amenable, residually finite group, then the
  equivalence relation of conjugacy on the set of free, Toeplitz
  $G$-subshifts is not hyperfinite.
\end{theorem}

The proof of this theorem proceeds by constructing a probability
measure $\mu$ on the set of Toeplitz subshifts which is invariant
under a suitable action of the group $G$, included in the equivalence
relation; then we show that the stabilizers of points are amenable and
conclude that the equivalence relation is not $\mu$-amenable and thus
not hyperfinite.

In the last section of the paper, reinterpreting a result of
Downarowicz, Kwiatkowski, and Lacroix, we indicate how topological
conjugacy of Toeplitz $\Z$-subshifts is naturally generated by an
action of a groupoid which is ``hyperfinite-by-compact.'' This seems
to indicate that the equivalence relation is somewhat simple; however,
we are not even able to prove that it is not universal. The following
question remains open:
\begin{question}
  Is topological conjugacy of Toeplitz $\Z$-subshifts hyperfinite? Is
  this true for an arbitrary residually finite, amenable group?
\end{question}

\subsection*{Acknowledgements} This work was initiated
during the stay of the first author at the University Paris
7. He would like to thank the logic group in Paris for their
hospitality. Both authors would like to thank Boban
Veli\v{c}kovi\'c for many valuable discussions. The first
author is also indebted to Tomasz Downarowicz for valuable
comments.


\section{Toeplitz subshifts}\label{sec:Toeplitz-subshifts}

In this section, we recall the definition and collect some basic
properties of Toeplitz subshifts for residually finite groups. We also
establish some basic definability properties that will be needed
later.

Let $G$ be a countable group. Recall that the \emph{profinite topology}
on $G$ is the one with the basis of neighborhoods at $1_G$ consisting
of all finite index subgroups (or, equivalently, all finite index
\emph{normal} subgroups). $G$ is called \emph{residually finite} if
$\set{1_G}$ is closed in the profinite topology, or, equivalently, if
the profinite topology is Hausdorff. For the rest of the paper, $G$
will always be a residually finite group.

The \emph{profinite completion} of $G$, denoted by $\widehat G$, is
the completion of the group uniformity defined by this topology;
equivalently $\widehat G = \varprojlim G/H$, where the limit is taken
over all finite index normal subgroups of $G$. This profinite
completion is metrizable if $G$ has only countably many subgroups of
finite index (for example, if it is finitely generated); even though
it is not strictly necessary for what we are doing (as we can always
pass to suitable metrizable quotients), we will sometimes assume this
for convenience.

Recall that the \emph{left shift action} $G \actson 2^G$ is defined by
$(g \cdot x)(h) = x(g^{-1}h)$. A closed subset $S \subseteq 2^G$ is
called a \textit{subshift} if it is invariant under this action.

\begin{definition}[Krieger \cite{krieger}]
  A word $x \in 2^G$ is called \textit{Toeplitz} if $x \colon G \to 2$
  is a continuous map for the profinite topology on $G$ (where
  $2 = \{0, 1\}$ is taken to be discrete). A subshift $S \subseteq2^G$
  is \textit{Toeplitz} if there exists a Toeplitz word $x$ such that
  $S=\overline{G \cdot x}$.
\end{definition}
Generalizing a well-known fact for $G = \Z$, Krieger showed that every
Toeplitz subshift is minimal (see \cite[Corollaire~2.5]{krieger}).

Let
\begin{equation*}
  S(G) = \{S \sub 2^G : S \text{ is closed and $G$-invariant}\}
\end{equation*}
be the set of all $G$-subshifts. This is naturally a compact space
with the Vietoris topology. It is easy to check that the family of
minimal subshifts as well as that of free subshifts form Borel subsets
of $S(G)$. In what follows, we see that being Toeplitz is also a Borel
condition and establish a simple selection lemma that will be used later.

We need the following well-known definability property of Baire
category notions for which we have not been able to find a suitable
reference. It is a slight generalization of \cite[16.1]{Kechris1995}
with the same proof. If $X$ is a Polish space, $F(X)$ denotes the
Effros Borel space of closed subsets of $X$. Recall that
$\exists^* x \in F$ means ``for non-meagerly many $x$ in $F$.''
\begin{proposition}
  \label{p:definability-Baire}
  \begin{enumerate}
  \item \label{i:1} Let $(X, \cS)$ be a measurable space, $Y$ a Polish space and let $\Phi
    \colon X \to F(Y)$ be a measurable function, where $F(Y)$ denotes
    the Effros Borel space of closed subsets of $Y$. Let $A
    \sub X \times Y$ be a measurable set (where $Y$ is equipped with its
    Borel $\sigma$-algebra) and $U \sub Y$ an open set. Then the set
    \begin{equation}
      \label{eq:non-mgr}
      \set{x \in X : U \cap \Phi(x) = \emptyset \text{ or } A_x \cap
        \Phi(x) \text{ is non-meager in } U \cap \Phi(x)}
    \end{equation}
    is measurable.
   
  \item \label{i:2} Let $X$ be a Polish space and $A \sub X$ be Borel. Then the set
  \begin{equation*}
    \set{F \in F(X) : \exists^* x \in F \ x \in A}
  \end{equation*}
  is Borel.
  \end{enumerate}
\end{proposition}
\begin{proof}
  \ref{i:1}. The proof goes exactly as in \cite[16.1]{Kechris1995}. If we let
  $A_U$ denote the set defined in \eqref{eq:non-mgr}, one checks that 
  \begin{itemize}
  \item for all $S \in \cS$, $V \sub Y$ open
    \begin{equation*}
      \begin{split}
      (S \times V)_U = \set{x \in X : U &\cap \Phi(x) = \emptyset
        \text{ or } \\
        & (x \in S \And U \cap V \cap \Phi(x) \neq \emptyset)};        
      \end{split}
    \end{equation*}
    
  \item $(\bigcup_n A_n)_U = \bigcup_n (A_n)_U$;
    
  \item $(\sim A)_U = \sim \bigcap_{U_n \sub U} A_{U_n}$, where the
    intersection is over all $U_n \sub U$ from a fixed countable base
    of $Y$.    
  \end{itemize}

  \ref{i:2}. This follows from \ref{i:1}.
\end{proof}

If $x$ is a Toeplitz word and $H \leq G$ is a finite index subgroup,
we let
\begin{equation*}
  \Per_H(x) = \set{g \in G : x|_{Hg} \text{ is constant}}.
\end{equation*}
The fact that $x$ is Toeplitz translates to $\bigcup_H \Per_H(x) = G$.
$H$ is called an \emph{essential group of periods} for $x$ if for all
$g \notin H$, $\Per_H(x) \nsubseteq \Per_H(g \cdot x)$.

The \emph{maximal equicontinuous factor} (\emph{m.e.f.} for short) of a
topological dynamical system $G \actson X$ is the factor generated by
all equicontinuous factors of $X$. For a Toeplitz subshift $S$, this
is always of the form $G \actson \varprojlim G/H_n$, where $\set{H_n}$
is a decreasing sequence of essential groups of periods (see
\cite[Proposition~7]{Cortez2008}). Note that every such system can
also be written as $G \actson \widehat G/K$, where
$K = \bigcap_n \cl{H_n}$ with the closures taken in $\widehat G$. We
have the following folklore lemma.
\begin{lemma}
  \label{l:Top-comeager}
  Let $S$ be a Toeplitz subshift. Then the set of Toeplitz words in
  $S$ is dense $G_\delta$.
\end{lemma}
\begin{proof}
  Let $Y$ be the m.e.f. of $S$ and $\pi \colon S \to Y$ be the factor
  map. By \cite[Theorem~2]{Cortez2008}, the set of Toeplitz words in
  $S$ can be written as
  \begin{equation*}
    \big\{ x \in S : \pi^{-1}(\set{\pi(x)}) = \set{x} \big\}.
  \end{equation*}
  This can be written as $\pi^{-1}(A)$, where
  \begin{equation*}
    A = \bigcap_{\eps > 0} \bigcup \set{ U \sub Y \text{ open}: \diam(\pi^{-1}(U)) < \eps}
  \end{equation*}
  and this is clearly $G_\delta$. It is also dense by the definition of
  a Toeplitz subshift.
\end{proof}

Denote
\begin{equation*}
  \Top(G) = \set{S \in S(G) : S \text{ is a Toeplitz subshift}}.
\end{equation*}
Lemma~\ref{l:Top-comeager} allows us to build the following selector
map that will be used throughout the paper.
\begin{proposition}
  \label{p:Toeplitz-selector}
  Let $G$ be a residually finite group. Then there exists a Borel map
  $\tau \colon \Top(G) \to 2^G$ such that for all $S \in \Top(G)$,
  $\tau(S) \in S$ and $\tau(S)$ is a Toeplitz word.
\end{proposition}
\begin{proof}
  This follows from Lemma~\ref{l:Top-comeager},
  Proposition~\ref{p:definability-Baire}, and
  \cite[18.6]{Kechris1995}.
\end{proof}

Now we can easily deduce the following.
\begin{lemma}
  \label{l:mef-Borel}
  Let $G$ be a residually finite group such that $\widehat G$ is
  metrizable. The map $\Top(G) \to F(\widehat G)$ which associates to
  a Toeplitz subshift $S$ (the conjugacy class of) the subgroup
  $K \leq \widehat G$ such that the m.e.f. of $S$ is isomorphic to
  $\widehat G/K$ is Borel.
\end{lemma}
\begin{proof}
  By Proposition~\ref{p:Toeplitz-selector}, for a given Toeplitz
  subshift $S$, we can choose in a Borel way a Toeplitz word
  $x \in S$. By \cite{Cortez2008}, we can choose in a Borel way a
  sequence $L_n(x)$ of finite index subgroups of $G$, which are
  essential periods of $x$ and such that, posing
  $K = \bigcap_n \cl{L_n}$, we have that the m.e.f. of $S$ is
  isomorphic to $\widehat G / K$. Now it is easy to check that the map
  that associates to the sequence $\set{L_n}_n$ the intersection
  $\bigcap_n \cl{L_n}$ is Borel.
\end{proof}

\begin{proposition}
  \label{p:Toeplitz-Borel}
  Let $G$ be a residually finite group, $\widehat G$ the profinite
  completion of $G$ and $K \leq \widehat G$ a closed subgroup. Then
  the set $\Top(G)$ is Borel and if $\widehat G$ is metrizable,
  \begin{equation*}
    \Top(G, K) = \set{S \in K(2^G) : S \text{ is a Toeplitz subshift with m.e.f. }
      \widehat G/K}
  \end{equation*}
  is also Borel.
\end{proposition}
\begin{proof}
  We have
  \begin{equation*}
      S \in \Top(G) \iff \exists^* x \in S \ x \text{ is Toeplitz}
  \end{equation*}
  and applying Proposition~\ref{p:definability-Baire} yields that
  $\Top(G)$ is a Borel set.

  The second statement follows from Lemma~\ref{l:mef-Borel}.
\end{proof}


\section{$\Z$-subshifts}
\label{sec:z-subshifts}

In this section, we consider $\Z$-subshifts and prove Theorem~\ref{th:z-subshifts}.

Let $x \in 2^\Z$ be a Toeplitz word and $p \in \N$. Denote by $\Per_p(x)$ 
the subset of $\Z / p\Z$ defined as follows:
\[
\Per_p(x) = \set{i + p \Z : x(i + kp) = x(i) \text{ for all } k \in \Z}.
\]
Let $H_p(x) = (\Z / p\Z) \setminus \Per_p(x)$. The elements of
$H_p(x)$ are called \emph{$p$-holes} for $x$.

We say that $x$ \emph{has separated holes} \cite{Downarowicz2005} if
\[
\lim_{p \to \infty} \min \set{|i - j| : i, j \in H_p(x), i \neq j} = \infty.
\]
Having separated holes is a property of the subshift that does not
depend on the choice of the word $x$. Thus, by the arguments in
Section~\ref{sec:Toeplitz-subshifts}, the condition of having
separated holes defines a Borel subset of the set of all subshifts.
All shifts with separated holes are regular and have topological
entropy $0$.

Next we recall the definition of an \emph{amenable} equivalence
relation. Let $E$ be a countable equivalence relation on the standard
Borel space $X$. If $x \in X$ and $f \colon E \to \R$ is a function,
denote by $f_x$ the function $[x]_E \to \R$ defined by
$f_x(y) = f(x, y)$. $E$ is called \emph{($1$-)amenable} if there exist
positive Borel functions $\lambda^n \colon E \to \R$ such that
\begin{itemize}
\item $\lambda^n_x \in \ell^1([x]_E)$, $\nm{\lambda^n_x}_1 = 1$;
  
\item $\lim_{n \to \infty} \nm{\lambda^n_x - \lambda^n_y}_1 = 0$ for
  all $(x, y) \in E$.
\end{itemize}
By a theorem of Connes, Feldman and Weiss~\cite{Connes1981} (see also
\cite{Kechris2004}), if $\mu$ is any probability measure on $X$ and
$E$ is amenable, then it is hyperfinite $\mu$-almost everywhere. It is
an open question whether every amenable equivalence relation is
hyperfinite. See \cite{jkl} for more details on amenable
equivalence relations in the Borel setting.

The goal of this section is to prove the following theorem.
\begin{theorem}
  \label{th:amenable}
  The equivalence relation of isomorphism of Toeplitz $\Z$-subshifts
  with separated holes is amenable.
\end{theorem}
\begin{proof} 
Let $p \in \N$. Denote by $\Sym(2^p)$ the set of all bijections
$2^p \to 2^p$. For $\pi \in \Sym(2^p)$, let
$\hat \pi \colon 2^\Z \to 2^\Z$ be defined by
$\hat \pi(x)|_{[kp, (k+1)p)]} = \pi(x|_{[kp, (k+1)p)})$ for all
$k \in \Z$. Then if $S \sub 2^\Z$ is a closed $p\Z$-invariant set,
$\hat \pi(S)$ is also a closed $p\Z$-invariant set and $\hat \pi$
defines an isomorphism between them (as $p\Z$-systems).

If $x \in 2^\Z$ and $i \in \Z$, denote by $x+i$ the shift of $x$ by
$i$. If $S \sub 2^\Z$ is a Toeplitz subshift, $x \in S$ is a Toeplitz word
and $p$ is a period, define $A_{p, i}(x) \sub S(\Z)$ and $B^p \colon E
\to \R$ by
\begin{align*}
A_{p, i}(x) &= \set{\cl{\Z \cdot \hat \pi(x+i)} : \pi \in \Sym(2^p)}, \\
B^p(S, T)  &= \sum_{i < p} \chi_{A_{p, i}(x)}(T),
\end{align*}
where $\chi_A$ denotes the characteristic function of $A$. Note that
the value of $B^p(S, T)$ does not depend on the choice of a Toeplitz
$x \in S$. Indeed, let $x_1, x_2 \in S$ be Toeplitz. We have that
$S = \bigcup_{j<p} \cl{p\Z \cdot (x_1+j)}$ and there exists $j_0$ such
that $x_2 \in \cl{p\Z \cdot (x_1+j_0)}$ and therefore (as
$\cl{p\Z \cdot (x_1+j_0)}$ is Toeplitz and thus, minimal),
$\cl{p\Z \cdot x_2} = \cl{p\Z \cdot (x_1 + j_0)}$. Then
$\cl{p\Z \cdot \hat \pi (x_2 + i)} = \cl{p\Z \cdot \hat \pi(x_1 + j_0
  + i)}$, whence, by minimality,
$\cl{\Z \cdot \hat \pi(x_2 + i)} = \cl{\Z \cdot \hat \pi(x_1 + j_0 +
  i)}$ for every $i$ and every $\pi$. Note finally that, as we can choose
$x = \tau(S)$ as per Proposition~\ref{p:Toeplitz-selector}, the
function $B^p$ is Borel.

Denote by $E$ the equivalence relation of isomorphism on Toeplitz
subshifts and define $\lambda^p \colon E \to \R$ by
\[
\lambda^p(S, T) = \frac{B^p(S, T)}{\sum_{T' \eqrel{E} S} B^p(S, T')} =
\frac{B^p(S, T)}{\nm{B^p_S}_1}.
\]
The functions $\lambda^p$ clearly satisfy the first condition in the
definition of amenability; in what follows, we check that if we
restrict to the subshifts with separated holes, they also satisfy the
second.

Recall that if $S$ and $T$ are subshifts and $f \colon S \to T$ is an
isomorphism, then it is given by a \emph{block code}, i.e., there
exists $r \in \N$ and a function $\phi \colon 2^{2r+1} \to 2$
such that for all $x \in S$,
\begin{equation*}
  f(x)(i) = \phi(x|_{[i - r, i + r]}).
\end{equation*}
The number $r$ is called the \emph{radius} of the block code.

Let $S$ and $T$ be two Toeplitz subshifts with separated holes and let
$f \colon S \to T$ be an isomorphism such that $f$ and $f^{-1}$ are
given by block codes of radius $r$. Let $\eps > 0$ be given. If
$C \sub \Z$, let
\[
(C)_r = \set{c + j \in \Z / p\Z : c \in C, |j| \leq r}.
\]
(Here and below, we suppress the natural quotient map $\Z \to \Z /
p\Z$ from the notation.) Note that $|(C)_r| \leq (2r+1)|C|$.

Let $x \in S$ be a Toeplitz word and let $y = f(x)$. Let $p$ be a
period so big that the distance between two consecutive holes in
$H_p(x)$ and in $H_p(y)$ is larger than $M = (2r+1)/\eps$.

\begin{claim} \label{cl:1}
For all $i \notin (H_p(x))_r$, we have $A_{p, i}(x) = A_{p, i}(y)$.
\end{claim}
\begin{proof}
  This follows from the fact that for $i \notin (H_p(x))_r$, there
  exists $\pi \in \Sym(2^p)$ such that $\hat \pi(x+i) = y+i$. Indeed,
  as $y = f(x)$ and $f$ is given by a block code of radius $r$,
  $y|_{[i+kp, i+(k+1)p)}$ only depends on $x|_{[i+kp-r, i+(k+1)p+r)}$
  and
  \begin{align*}
    x|_{[i+kp-r, i+kp)]} &= x|_{[i+(k+1)p-r, i+(k+1)p)} \\
    x|_{[i+(k+1)p, i+(k+1)p+r)} &= x|_{[i+kp, i+kp+r)}
  \end{align*}
  because $(i-r, i+r) \sub \Per_p(x)$. Thus $y|_{[i+kp, i+(k+1)p)}$ can
  be calculated from $x|_{[i+kp, i+(k+1)p)}$ in a way independent of
  $k$ which gives the claim.
\end{proof}

\begin{claim} \label{cl:2}
For all $i \notin H_p(x)$, we have $A_{p, i}(x) = A_{p, i+1}(x)$.
\end{claim}
\begin{proof}
Indeed, let $\sigma \colon 2^{\Z / p\Z} \to 2^{\Z / p\Z}$ be defined by
$\sigma(z)(j) = z(j-1)$, and observe that as $i \in \Per_p(x)$,
$\hat \sigma (x+i) = x+i+1$. Then
\begin{equation*}
  \begin{split}
    A_{p, i}(x) &= \set{\cl{\Z \cdot \hat \pi(x+i)} : \pi \in
      \Sym(2^p) } \\
    &= \set{\cl{\Z \cdot (\widehat{\pi \sigma})(x+i)} : \pi \in
      \Sym(2^p)} \\
    &= A_{p, i+1}(x).
    \qedhere
  \end{split} 
\end{equation*}
\end{proof}

Let $h_0, h_1, h_2, \ldots h_J$ enumerate the holes in $H_p(x)$ in
circular order (so that $J+1 = 0$). By the choice of $p$,
$|h_{j} - h_{j+1}| \geq M$ for all $j$. If $f$ is a real-valued function,
denote by $f^+$ the function $\max(f, 0)$ and note that $\nm{f}_1 =
\nm{f^+}_1 + \nm{(-f)^+}_1$. We calculate:
\[
\begin{split}
  \nm{(B^p_S - B^p_T)^+}_1 &= \nm[\big]{\big(\sum_{i < p} \chi_{A_{p, i}(x)} -
  \sum_{i < p} \chi_{A_{p, i}(y)} \big)^+ }_1 \\
  &\leq \nm[\big]{\sum_{j < J} \sum_{i = h_j-r}^{h_j+r} \chi_{A_{p, i}(x)}}_1  \\
  &\leq \nm[\big]{\sum_{j < J} \big( \sum_{i = h_j+1}^{h_j + r}
  \chi_{A_{p, i}(x)} + \sum_{i = h_{j+1} - r}^{h_{j+1}} \chi_{A_{p,i}(x)} \big)}_1 \\
  &\leq \sum_{j < J} \big( (2r+1)/M \big) \sum_{i = h_j+1}^{h_{j+1}} |A_{p,i}(x)| \\
  &= \big((2r+1)/M \big) \nm{B^p_S}_1 \\
  &\leq \eps \nm{B^p_S}_1.
\end{split}
\]
For the inequality on the second line, we use Claim~\ref{cl:1}, and for the one
on the fourth line, we use Claim~\ref{cl:2}. Similarly,
$\nm{(B^p_T - B^p_S)^+}_1 \leq \eps \nm{B^p_T}_1$, whence
\begin{equation*}
\nm{B^p_S - B^p_T}_1 \leq \eps \big( \nm{B^p_S}_1 + \nm{B^p_T}_1 \big).  
\end{equation*}
Using these estimates, it is now easy to see that if $\set{p_n}$ is
a sequence of periods such that $p_n | p_{n+1}$ for all $n$, then
$\nm{\lambda^{p_n}_x - \lambda^{p_n}_y}_1 \to 0$ as $n \to \infty$.
\end{proof}


\section{The non-amenable case}
\label{sec:non-amenable-case}

Now we return to the general situation of
Section~\ref{sec:Toeplitz-subshifts}, where $G$ is a residually finite
group. We also fix a decreasing sequence $\set{H_n}$ of finite index,
\emph{normal} subgroups with trivial intersection, to be determined
later. We also denote by $\hat G$ the inverse limit
$\varprojlim G/H_n$ (which is a group quotient of the profinite
completion $\widehat G$) and by $\pi_n \colon \hat G \to G/H_n$ the
natural projections.

The goal of this section is to prove Theorem~\ref{th:non-amenable}.
The proof proceeds by constructing a probability measure on the set of
Toeplitz subshifts which is invariant under an appropriate action of
$G$ contained in the isomorphism relation. We then check that the
point stabilizers for this action are amenable and conclude that if
$G$ is non-amenable, then the isomorphism equivalence relation is not
amenable either.

\subsection{The left and right actions}\label{sec:actions}
First, we describe the construction of the measure. Let
$A_n = (H_{n-1}/H_n) \setminus \{H_n\}$ and let
$Z = 2^{\bigsqcup_n A_n}$. Let
$Y = \{y \in \hat G : y \notin G\}$. Consider the maps
\[
Y \times Z \xrightarrow{\ \sigma\ } 2^G \xrightarrow{\ \rho\ } S(G)
\]
defined as follows. If $(y, z) \in Y \times Z$, define
$\sigma(y, z)$ by
\[
\sigma(y, z)(h) = z(\pi_{n_0}(y^{-1}h)), \quad \text{where } n_0 = \min
\{n : \pi_n(y) \neq \pi_n(h)\}.
\]
Define $\rho \colon 2^G \to S(G)$ by $\rho(x) = \overline{G \cdot
x}$. Let $\theta = \rho \circ \sigma$.

$G$ acts on $2^G$ on both sides: on the left, by
\[
(g \cdot x)(h) = x(g^{-1}h)
\]
and on the right, by
\[
(x \cdot g)(h) = x(hg^{-1})
\]
and the two actions commute. Note that if $S \sub 2^G$ is a subshift
and $g \in G$, then $S \cdot g$ is a subshift too, so we have a right
action $S(G) \curvearrowleft G$. Moreover, for every fixed $g \in G$,
the map
\[
S \to S \cdot g, \quad x \mapsto x \cdot g
\]
is an isomorphism of subshifts.

The space $Y \times Z$ is also equipped with commuting left and right
actions of $G$, defined as follows. For $y \in Y$ and $g \in G$,
define $g \cdot y = gy$ and $y \cdot g = yg$ (recall that
$G \subseteq \hat G$). For $z \in Z$, define $g \cdot z = z$ and
$(z \cdot g)(a) = z(gag^{-1})$ for all $a \in \bigsqcup_n A_n$. Equip
$Y \times Z$ with the diagonal left and right actions.

\begin{lemma}\label{l:invariance}
  The maps $\sigma$ and $\rho$ commute with both the left and right
  actions of $G$ on the respective spaces.
\end{lemma}
\begin{proof}
  We only check that $\sigma$ commutes with the right actions. Suppose
  that $(y_1,z_1),(y_2,z_2)$ are elements of $Y\times Z$ such that
  $(y_1,z_1)\cdot g=(y_2,z_2)$. Let $x_1=\sigma(y_1,z_1)$ and
  $x_2=\sigma(y_2,z_2)$. Note that for every $h\in G$ we have
  \begin{equation*}
  \min\{n:\pi_n(y_1)\not=\pi_n(hg^{-1})\}=\min\{n:\pi_n(y_1g)\not=\pi_n(h)\}.
  \end{equation*}
  Write $n_0(h)$ for this common value. Thus,
  \begin{align*}
    (x_1\cdot g)(h) &=
                      x_1(hg^{-1})=z_1(\pi_{n_0(h)}(y_1^{-1}hg^{-1})),
                      \quad \text{ and } \\
    x_2(h) &= z_2(\pi_{n_0(h)}(y_2^{-1}h))=z_2(\pi_{n_0(h)}(g^{-1}y_1^{-1}h)).
  \end{align*}
  But
  \begin{equation*}
    \begin{split}
      z_2(\pi_{n_0(h)}(g^{-1}y_1^{-1}h)) &=
        z_1(g\pi_{n_0(h)}(g^{-1}y_1^{-1}h)g^{-1})) \\
      &= z_1(\pi_{n_0(h)}(y_1^{-1}hg^{-1})),    
    \end{split}
  \end{equation*}
  so $x_1\cdot g=x_2$, as needed.
\end{proof}

We call an element $z \in Z$ \textit{proper} if it takes both values
$0$ and $1$ infinitely many times. Note that $z$ is proper if and only
if for every (any) $y \in Y$, $G \cdot \sigma(y, z)$ is infinite.
\begin{lemma}\label{l:mef}
   For any $y\in Y$ and proper $z\in Z$, $\sigma(y, z)$ is a Toeplitz
   word and the m.e.f. of $\cl{G \cdot \sigma(y, z)}$ is isomorphic to
   $\hat G$. In particular, the subshift $\cl{G \cdot \sigma(y, z)}$ is free.
\end{lemma}
\begin{proof}
  Write $x=\sigma(y,z)$. First note that for any $g \in G$, the value
  $x(g)$ is assumed on the whole $H_{n_0}$-coset $\pi_{n_0}(y^{-1}g)$,
  where $n_0 = \min \{n : \pi_n(y) \neq \pi_n(g)\}$, showing that $x$
  is Toeplitz.

  To calculate the m.e.f., by \cite{Cortez2008}, it suffices to observe
  that for all $n$, $H_n$ is an essential group of periods for $x$.
  Indeed, using the fact that $z$ is proper, we have that
  $\Per_{H_n}(x) = G \setminus \pi_n(y)$ which is not contained in
  $\Per_{H_n}(g \cdot x) = G \setminus \pi_n(g \cdot y)$ for any $g
  \notin H_n$.

  For the final claim, just note that as $\bigcap_n H_n = \set{1_G}$
  and the $H_n$ are normal, the translation action $G \actson \hat G$
  is free.
\end{proof}

The main ingredient for the proof of the theorem is the following lemma.
\begin{lemma}
  \label{l:injective}
  For all proper $z_1, z_2 \in Z$ and all $y_1, y_2 \in Y$,
  \[
  z_1 \neq z_2 \implies \theta(y_1, z_1) \neq \theta(y_2, z_2).
  \]
\end{lemma}
\begin{proof}
  Let $x_i = \sigma(y_i, z_i)$. To show that $\cl{G \cdot
    x_1} \neq \cl{G \cdot x_2}$, it suffices to find a
  ``subword'' of $x_2$ that does not occur in $x_1$, i.e., a
  finite set $F \subseteq G$ such that
  \[
  \forall g \in G \ \exists f \in F \quad x_1(gf) \neq x_2(f).
  \]
  
  Let $a_0 \in A_n$ be such that $z_1(a_0) \neq
  z_2(a_0)$. By replacing $(y_1, z_1)$ and $(y_2, z_2)$ with
  $g_1 \cdot (y_1, z_1)$ and $g_2 \cdot (y_2, z_2)$ (which
  does not change either $z_i$ or $\theta(y_i, z_i)$) for
  suitably chosen $g_1, g_2 \in G$, we may assume that
  $\pi_n(y_1) = \pi_n(y_2) = 1_{G/H_n}$. Let $f_0, f_1 \in
  G$ be such that $\pi_n(f_0) = \pi_n(f_1)$, $x_2(f_0) = 0$,
  $x_2(f_1) = 1$ (such $f_0, f_1$ exist by the assumption
  that $z_2$ is proper). Let $f \in G$ be such that
  $\pi_n(f) = a_0$ and finally, let $F = \{f, f_0, f_1\}$.

  Next we show that this $F$ works. Let $g \in G$ be
  arbitrary. We distinguish two cases: $g \in H_n$ and $g
  \notin H_n$. Let first $g \in H_n$. We check that $x_1(gf)
  \neq x_2(f)$. Indeed,
  \[
  x_1(gf) = z_1(\pi_n(y_1^{-1}gf)) = z_1(\pi_n(f)) = z_1(a_0)
  \]
  and
  \[
  x_2(f) = z_2(\pi_n(y_2^{-1}f)) = z_2(\pi_n(f)) = z_2(a_0),
  \]
  which are different by the choice of $a_0$.

  Suppose now that $g \notin H_n$. We will show that
  $x_1(gf_0) = x_1(gf_1)$ which will complete the proof (as
  $x_2(f_0) \neq x_2(f_1)$). Indeed, observe that the least
  $k$ for which $\pi_k(gf_i) \neq 1_{G/H_k}$ is at most $n$
  and is the same for $i = 0, 1$ (as $\pi_n(f_0) =
  \pi_n(f_1)$). Recall also that $\pi_k(y_1) = \pi_k(y_2) =
  1_{G/H_k}$. Now we have
  \[
    \begin{split}
      x_1(gf_0) &= z_1(\pi_k(y_1^{-1}gf_0)) = z_1(\pi_k(gf_0)) \\
      &= z_1(\pi_k(gf_1)) = x_1(gf_1),      
    \end{split}
  \]
  completing the proof.
\end{proof}

\subsection{A.e. amenable stabilizers}\label{sec:stabilizers}
Let $\lambda$ be the Haar measure on $\hat G$ and note that as
$\lambda(Y) = 1$, we can consider $\lambda$ as a measure on $Y$. Let
$\nu$ be the Bernoulli $(\frac{1}{2},\frac{1}{2})$ measure on
$Z = 2^{\bigsqcup_n A_n}$. Equip $Y \times Z$ with the product measure
$\lambda \times \nu$. It is clear that this measure is invariant under
both the left and the right action of $G$ (but for us it will be the
right one that will be important). Let $\mu = \theta_*(\lambda \times \nu)$.

\begin{lemma}\label{l:mu-concentrates}
  The measure $\mu$ concentrates on the set of free, Toeplitz
  $G$-subshifts and is invariant under the right action.
\end{lemma}
\begin{proof}
  This follows from Lemma~\ref{l:mef} and the fact that $\nu$
  concentrates on the set of proper elements of $Z$. Invariance follows
  directly from Lemma~\ref{l:invariance}.
\end{proof}

For $g \in G$, denote by $C(g)$ the \emph{centralizer} of $g$ in $G$,
i.e., the set of all elements of $G$ that commute with $g$. Let
\[
Z_f(G) = \{g\in G: C(g)\text{ has finite index in }G\}.
\]

\begin{lemma}\label{l:Zf-amenable}
  $Z_f(G)$ is an amenable, normal subgroup of $G$.
\end{lemma}
\begin{proof}
  It is clear that $Z_f(G)$ is a normal subgroup of $G$. To
  see that it is amenable, note that the centralizers of all
  elements in $Z_f(G)$ have finite index in $Z_f(G)$. Hence,
  $Z_f(G)$ has finite conjugacy classes and this implies
  amenability (Leptin~\cite{Leptin1968}).
\end{proof}

Next we choose a suitable sequence $\set{H_n}$ of finite index, normal
subgroups of $G$ that will allow us to prove the theorem. Note that
the residual finiteness of $G$ implies that for any finite index
subgroup $H\lhd G$ and $g\in G \setminus Z_f(G)$ there exists a finite
index subgroup $H' \leq H$ normal in $G$ such that for some
$C\in H/H'$ we have $gCg^{-1} \neq C$. Enumerate $G \setminus Z_f(G)$
as $\{g_n : n \in \N \}$ and construct inductively $\set{H_n:n\in\N}$
a decreasing sequence with trivial intersection such that
\begin{equation} \label{eq:Hn-conditions}
  \text{for all $m \geq n$ there exists $C \in H_m/H_{m+1}$ such that
  $g_n^{-1}Cg_n \neq C$}.
\end{equation}

\begin{proof}[Proof of Theorem~\ref{th:non-amenable}]
  We will show that the right action of $G$ on $\Top(G)$ has $\mu$-a.e
  amenable stabilizers. Recall that (see, e.g.,
  \cite[Lemma~3.6]{Hjorth2008}) if a measure-preserving action of a
  group has amenable stabilizers and induces a hyperfinite equivalence
  relation, then the group is amenable. Thus, if $G$ is non-amenable,
  we obtain that the orbit equivalence relation induced by the right
  action of $G$ on the set of Toeplitz $G$-subshifts is not
  hyperfinite and, as it is contained in the topological conjugacy
  relation, the latter is not hyperfinite either.

  By Lemma~\ref{l:Zf-amenable}, it is enough to show that there is a
  measure $1$ subset $A$ of $\Top(G)$ such that the stabilizer of
  every element in $A$ is contained in $Z_f(G)$. Since $G$ is
  countable, it suffices to see that for every
  $g \in G \setminus Z_f(G)$, the set
  \begin{equation*}
  \{(y,z) \in Y \times Z : \theta(y,z) \cdot g \neq \theta(y,z)\}
  \end{equation*}
  has measure $1$. By Lemma~\ref{l:injective}, it is enough to show
  that $\{z \in Z : z \cdot g \neq z \}$ is of measure $1$. As
  $g \notin Z_f(G)$, by \eqref{eq:Hn-conditions}, for almost all $n$,
  there is an element, say $C_n \in H_{n+1} / H_n$ such that
  $gC_ng^{-1} \neq C_n$. By the definition of the action of $G$ on
  $Z$, for any $z \in Z$ such that there exists $n$
  with $z(C_n) \neq z(g C_n g^{-1})$, we have $z \neq z \cdot g$. But the
  latter condition is clearly satisfied on a measure $1$ set. This
  completes the proof of the theorem.
\end{proof}

\begin{remark}
  Note that in case $G$ is a non-cyclic free group, the right action
  of $G$ on $\Top(G)$ is free $\mu$-a.e. because $Z_f(G)$ is trivial
  (centralizers of non-trivial elements of free groups are cyclic).
  This implies in particular that the equivalence relation of
  topological conjugacy of Toeplitz $G$-subshifts embeds a free,
  measure-preserving action of a free group.
\end{remark}


\section{The groupoid viewpoint}
\label{sec:groupoid-viewpoint}

The equivalence relation of isomorphism of subshifts is
naturally given by an action of a \emph{groupoid} rather
than a group. Recall that a \emph{groupoid} is a small
category where each arrow is invertible. We will only be
interested in countable Borel groupoids, defined as follows
(see also \cite{martino}). A \emph{countable Borel groupoid
  $\Gamma$} is a tuple $(X, A, s, r, \circ)$, where $X$ is a
standard Borel space of objects, $A$ is a standard Borel
space of arrows, $s \colon A \to X$ is a Borel map
specifying the \emph{source} of each arrow and $r \colon A
\to X$ a Borel map specifying its \emph{range}. $\circ$ is a
partial Borel map $A \times A \to A$ which represents the
composition of arrows; $f \circ g$ is defined whenever $r(g)
= s(f)$.  Of course, we require that $(X, A, \circ)$ be a
groupoid. Furthermore, we require that $r^{-1}(\set{x})$ and
$s^{-1}(\set{x})$ be countable sets for every $x \in X$. By
the Lusin--Novikov selection theorem, there exist sequences
of Borel maps $P_n \colon X \to A$ and $Q_n \colon X \to A$
such that $P_n(x)$ enumerate $s^{-1}(x)$ and $Q_n(x)$
enumerate $r^{-1}(x)$ for every $x$. If $x, y \in X$, we
will denote by $A(x, y)$ the set of arrows between $x$ and
$y$. We will sometimes identify a groupoid with the set of
its arrows as the other information can be recovered from
it.

Every countable Borel groupoid $\Gamma$ gives rise to a countable
Borel equivalence relation $E_\Gamma$ on $X$ defined by
\begin{equation*}
  x \eqrel{E_\Gamma} y \iff \exists f \in A \ s(f) = x \And r(f) = y.
\end{equation*}
Conversely, every countable equivalence relation can be viewed as
groupoid, where there is a unique arrow between every pair of
equivalent elements of $X$.

Now let $G$ be a residually finite group as before, $\set{H_n}$ be a
sequence of finite index, normal subgroups of $G$, and let $X$ be the
standard Borel space of of all Toeplitz subshifts of $G$ with m.e.f.
$\hat G = \varprojlim G/H_n$. Let $A$ be the set of all
isomorphisms between elements of $X$. Recalling that every such
isomorphism is given by a block code (which is a finite object), it is
easy to endow $A$ with a standard Borel structure so that
$(X, A, \circ)$ becomes a countable Borel groupoid ($\circ$ is just
composition of maps). The equivalence relation generated by this
groupoid is exactly isomorphism of subshifts. (All of this is defined
for arbitrary subshifts of countable groups; however, below we will
need that they be Toeplitz.)

Note that for every $S \in X$ and every $x \in S$, there exists a
unique factor map $\pi_x \colon S \to \hat G$ such that
$\pi(x) = 1_{\hat G}$. (Existence follows from the fact that if
$\pi \colon S \to \hat G$ is any $G$-map, then post-composing $\pi$
with right multiplication by $\pi(x)^{-1}$ yields a map that sends $x$
to $1_{\hat G}$.) Let $\tau \colon X \to 2^G$ be a Borel map that
selects for every $S \in X$, a Toeplitz word $\tau(S) \in S$ as given
by Proposition~\ref{p:Toeplitz-selector}. Now we can define a Borel
cocycle $\alpha_0 \colon \Gamma \to \hat G$ by
\begin{equation*}
  \alpha_0(f) = \pi_{\tau(r(f))}(f(\tau(s(f))))^{-1}.
\end{equation*}
Recall that a \emph{cocycle} is just a map $\Gamma \to \hat G$ that
satisfies $\alpha(f \circ g) = \alpha(f)\alpha(g)$. Two cocycles
$\alpha, \beta$ are \emph{cohomologous} if there exists a Borel map
$F \colon X \to \hat G$ such that
$\beta(f) = F(r(f))^{-1} \alpha(f) F(s(f))$ for all $f \in \Gamma$.
Note that changing the selection map $\tau$ transforms $\alpha_0$ into
a cohomologous cocycle. We also have a natural Borel homomorphism
$\rho \colon \Gamma \to E$ defined by
\begin{equation*}
  \rho(f) = (s(f), r(f)).
\end{equation*}




Now we specialize to the case $G = \Z$. We have the following
proposition, which basically follows from a result by Downarowicz,
Kwiatkowski, and Lacroix~\cite{dkl}.
\begin{proposition}
  Let $\Delta = \ker \alpha_0 = \set{f \in \Gamma : \alpha_0(f) = 0}$.
  Then the equivalence relation $E_\Delta$ is hyperfinite.
\end{proposition}
\begin{proof}
  For a period $p$, define the finite equivalence relation $E_p$ on
  $\Top(\Z)$ by
  \begin{equation*}
    S \eqrel{E_p} T \iff \exists \sigma \in \Sym(2^p) \ \hat
    \sigma(\tau(S)) = \tau(T).
  \end{equation*}
  It follows from \cite[Theorem~1]{dkl} that
  \begin{equation*}
    \begin{split}
    E_\Delta &= \set{(S, T) : \exists f \colon S \to T \text{
        isomorphism such that } f(\tau(S)) = \tau(T)} \\
    &= \bigcup_p E_p,
    \end{split}
  \end{equation*}
  showing that $E_\Delta$ is hyperfinite.
\end{proof}

If $S$ is a subshift, the \emph{centralizer} $C(S)$ of $S$ is the
group of automorphisms of $S$, or equivalently, the group of arrows
from $S$ to itself. Note that by our observations above, if $S \in X$,
then $C(S)$ embeds in $\hat G$ and is therefore an abelian group. So,
in some sense, the grooupoid $\Gamma$ differs from the equivalence
relation $E_\Gamma$ only a little.

We finally observe that the existence of the cocycle
$\alpha_0$ gives some restrictions on the groupoid
$\Gamma$. For example, using Popa's cocycle superrigidity
results \cite{popa}, it is easy to prove that $\Gamma$ does
not embed the groupoid of the free part of any Bernoulli
action of an infinite property (T) group. However, it is not
clear how to conclude anything from that about the
equivalence relation $E$; in particular, we do not know
whether $E$ is universal.

\bibliography{refs}

\begin{thebibliography}{10}

\bibitem{clemens}
John~D. Clemens.
\newblock Isomorphism of subshifts is a universal countable {B}orel equivalence
  relation.
\newblock {\em Israel J. Math.}, 170:113--123, 2009.

\bibitem{Connes1981}
Alain Connes, Jacob Feldman, and Benjamin Weiss.
\newblock An amenable equivalence relation is generated by a single
  transformation.
\newblock {\em Ergodic Theory Dynamical Systems}, 1(4):431--450 (1982), 1981.

\bibitem{Cortez2008}
Mar{\'{\i}}a~Isabel Cortez and Samuel Petite.
\newblock {$G$}-odometers and their almost one-to-one extensions.
\newblock {\em J. Lond. Math. Soc. (2)}, 78(1):1--20, 2008.

\bibitem{Cortez2014}
Mar{\'{\i}}a~Isabel Cortez and Samuel Petite.
\newblock Invariant measures and orbit equivalence for generalized {T}oeplitz
  subshifts.
\newblock {\em Groups Geom. Dyn.}, 8(4):1007--1045, 2014.

\bibitem{djk}
R.~Dougherty, S.~Jackson, and A.~S. Kechris.
\newblock The structure of hyperfinite {B}orel equivalence relations.
\newblock {\em Trans. Amer. Math. Soc.}, 341(1):193--225, 1994.

\bibitem{dkl}
T.~Downarowicz, J.~Kwiatkowski, and Y.~Lacroix.
\newblock A criterion for {T}oeplitz flows to be topologically isomorphic and
  applications.
\newblock {\em Colloq. Math.}, 68(2):219--228, 1995.

\bibitem{downarowicz}
Tomasz Downarowicz.
\newblock The {C}hoquet simplex of invariant measures for minimal flows.
\newblock {\em Israel J. Math.}, 74(2-3):241--256, 1991.

\bibitem{Downarowicz2005}
Tomasz Downarowicz.
\newblock Survey of odometers and {T}oeplitz flows.
\newblock In {\em Algebraic and topological dynamics}, volume 385 of {\em
  Contemp. Math.}, pages 7--37. Amer. Math. Soc., Providence, RI, 2005.

\bibitem{gjs}
S.~Gao, S.~Jackson, and B.~Seward.
\newblock {\em Group colorings and {B}ernoulli subflows}.
\newblock Memoirs of the American Mathematical Society.
\newblock to appear.

\bibitem{gps}
Thierry Giordano, Ian~F. Putnam, and Christian~F. Skau.
\newblock Full groups of {C}antor minimal systems.
\newblock {\em Israel J. Math.}, 111:285--320, 1999.

\bibitem{Hjorth2008}
Greg Hjorth.
\newblock Non-treeability for product group actions.
\newblock {\em Israel J. Math.}, 163:383--409, 2008.

\bibitem{hk}
Greg Hjorth and Alexander~S. Kechris.
\newblock Rigidity theorems for actions of product groups and countable {B}orel
  equivalence relations.
\newblock {\em Mem. Amer. Math. Soc.}, 177(833):viii+109, 2005.

\bibitem{jkl}
S.~Jackson, A.~S. Kechris, and A.~Louveau.
\newblock Countable {B}orel equivalence relations.
\newblock {\em J. Math. Log.}, 2(1):1--80, 2002.

\bibitem{jm}
Kate Juschenko and Nicolas Monod.
\newblock Cantor systems, piecewise translations and simple amenable groups.
\newblock {\em Ann. of Math. (2)}, 178(2):775--787, 2013.

\bibitem{Kechris1995}
Alexander~S. Kechris.
\newblock {\em Classical descriptive set theory}, volume 156 of {\em Graduate
  Texts in Mathematics}.
\newblock Springer-Verlag, New York, 1995.

\bibitem{Kechris2004}
Alexander~S. Kechris and Benjamin~D. Miller.
\newblock {\em Topics in orbit equivalence}, volume 1852 of {\em Lecture Notes
  in Mathematics}.
\newblock Springer-Verlag, Berlin, 2004.

\bibitem{krieger}
Fabrice Krieger.
\newblock Sous-d\'ecalages de {T}oeplitz sur les groupes moyennables
  r\'esiduellement finis.
\newblock {\em J. Lond. Math. Soc. (2)}, 75(2):447--462, 2007.

\bibitem{Leptin1968}
H.~Leptin.
\newblock Zur harmonischen {A}nalyse klassenkompakter {G}ruppen.
\newblock {\em Invent. Math.}, 5:249--254, 1968.

\bibitem{lind.marcus}
Douglas Lind and Brian Marcus.
\newblock {\em An introduction to symbolic dynamics and coding}.
\newblock Cambridge University Press, Cambridge, 1995.

\bibitem{martino}
Martino Lupini.
\newblock Polish groupoids and functorial complexity.
\newblock 2014.
\newblock preprint, arxiv.org/abs/1407.6671.

\bibitem{matui}
Hiroki Matui.
\newblock Some remarks on topological full groups of {C}antor minimal systems.
\newblock {\em Internat. J. Math.}, 17(2):231--251, 2006.

\bibitem{popa}
Sorin Popa.
\newblock Cocycle and orbit equivalence superrigidity for malleable actions of
  {$w$}-rigid groups.
\newblock {\em Invent. Math.}, 170(2):243--295, 2007.

\bibitem{thomas}
Simon Thomas.
\newblock Topological full groups of minimal subshifts and just-infinite
  groups.
\newblock In {\em Proceedings of the 12th {A}sian {L}ogic {C}onference}, 2012.

\bibitem{jay}
Jay Williams.
\newblock Isomorphism of finitely generated solvable groups is weakly
  universal.
\newblock {\em Journal of Pure and Applied Algebra}.
\newblock to appear.

\end{thebibliography}
\end{document}